\documentclass[leqno,12pt]{amsart} 
\usepackage[top=30truemm,bottom=30truemm,left=25truemm,right=25truemm]{geometry}
\usepackage{amssymb}
\usepackage{amsmath}
\usepackage{amsthm}
\usepackage{amscd}
\usepackage{mathrsfs}
\usepackage{graphicx}
\usepackage[dvips]{color}
\usepackage[all]{xy}

\usepackage{url}
\usepackage{comment}
\theoremstyle{plain} 
\newtheorem{theorem}{\indent\bf Theorem}[section]
\newtheorem{lemma}[theorem]{\indent\bf Lemma}

\newtheorem{proposition}[theorem]{\indent\bf Proposition}

\theoremstyle{definition} 

\newtheorem{remark}[theorem]{\indent\bf Remark}

\newcommand{\dbar}{\bar{\partial}}

\makeatletter

\@addtoreset{equation}{section}
\makeatother
\begin{document}
	
	\title[]{A simplified proof of optimal $L^2$-extension theorem and extensions from non-reduced subvarieties}
	
	\author[G. Hosono]{Genki Hosono} 
	
	%
	\address{
		Mathematical Institute, Tohoku University, 6-3, Aramaki Aza-Aoba, Aoba-ku, Sendai, 980-8578, Japan
	}
	\email{genki.hosono.a4@tohoku.ac.jp, genki.hosono@gmail.com}
	\maketitle
\begin{abstract}
We give a simplified proof of an optimal version of the Ohsawa-Takegoshi $L^2$-extension theorem. We follow the variational proof by Berndtsson-Lempert and use the method in the paper of McNeal-Varolin. As an application, we give an optimal estimate for extensions from possibly non-reduced subvarieties.
\end{abstract}
	
\section{Introduction}
The goals of this paper are to provide a simplified proof of an optimal version of the Ohsawa-Takegoshi $L^2$-extension theorem and to apply the same method to extensions from possibly non-reduced subvarieties to get an optimal estimate.

\textit{The Ohsawa-Takegoshi $L^2$-extension theorem} \cite{OT} is an extension theorem for holomorphic functions defined on a closed subvariety of a complex manifold with an $L^2$-estimate (For precise statement, see Theorem \ref{thm:main} below for example). This theorem is now widely used in complex and algebraic geometry.  
Recently, {\it an optimal $L^2$-extension theorem} was obtained by B{\l}ocki \cite{Blo} and Guan and Zhou \cite{GZ}. Their proofs are based on a version of H\"ormander's $L^2$-estimate for $\dbar$-equation and a careful choice of auxiliary functions. After that Berndtsson and Lempert obtained a new proof for the optimal $L^2$-extension theorem \cite{BL} based on a variational result of Berndtsson \cite{Ber}.

In this paper, we will give a simplified proof of the $L^2$-extension theorem along with the proof in \cite{BL}.
We simplified a limiting argument in their proof (in particular, Lemma 3.5 in \cite{BL}) by using a method by McNeal and Varolin \cite{MV}.
In \cite{MV}, the limiting argument is replaced by an existence result of an extension to a small neighborhood of the subvariety with a certain $L^2$-estimate (Theorem 4.5 in \cite{MV}). 
Since they considered the jet extension, in \cite{MV}, $\dbar$-equation was solved to get an extension of jets to a small neighborhood. In our setting, it is enough to consider the restriction of a fixed arbitrary extension.

Moreover, by the same method, we will give an optimal $L^2$-estimate for extensions from non-reduced subvarieties in a formulation similar to \cite{Dem}. 
Let $\Omega \Subset \mathbb{C}^n$ be a bounded pseudoconvex domain and $\phi \in PSH(\Omega) \cap C(\Omega)$ be a continuous plurisubharmonic function.
Let $\psi$ be a negative plurisubharmonic function with neat analytic singularities (see Section 3 for definitions).
Let $0=m_0<m_1<\cdots < m_p <\cdots$ be the jumping numbers of $\psi$. Then we have:
\begin{theorem}\label{thm:non-reduced}
	Let $f \in H^0(\Omega, \mathcal{I}(m_{p-1}\psi) /\mathcal{I}(m_p\psi))$. Assume that there exists a holomorphic function $F^\circ$ on $\Omega$ such that $F^\circ|_{\mathcal{O}/\mathcal{I}(m_p \psi)} = f$ and $\int_\Omega |F^\circ|^2 e^{-\phi - m_{p-1}\psi} < +\infty$. Then the $L^2$-minimal extension $F_0 \in A^2(\Omega, \phi + m_{p-1} \psi)$ of $f$ satisfies
	$$\int_\Omega |F_0|^2 e^{-\phi - m_{p-1} \psi} \leq \limsup_{t \to -\infty} e^{-(m_p - m_{p-1})t} \int_{\psi<t} |F^\circ|^2 e^{-\phi - m_{p-1} \psi}. $$
\end{theorem}

Note that the right-hand side is slightly different from the {\it $m_p$-jet $L^2$-norm} $|J^{m_p} f|^2_{\omega,h} dV_{Z^\circ_p, \omega} [\psi]$, which was used in \cite{Dem}. The relationship of these norms for jets are to be studied.

Very recently, Zhou and Zhu \cite{ZZ} proved an optimal extension theorem in a more general setting. Their proof is based on a version of $L^2$-estimates for $\dbar$-equations. It may be interesting to know if variational methods can be applied to this general setting.
	
\section{Simplest case: a simplification of Berndtsson-Lempert's proof}
To make the idea clear, first we will explain the simplest case. In this case we obtain a simplified proof of an optimal $L^2$-extention theorem. The following version is in \cite{BL}:

\begin{theorem}[Optimal $L^2$ extension, \cite{Blo}, \cite{GZ}, \cite{BL}]\label{thm:main}
	Let $\Omega \subset \mathbb{C}^n$ be a bounded pseudoconvex domain and $V \subset \Omega$ be a closed submanifold. Let $\phi \in PSH(\Omega)$. Assume that there exists a \textit{Green-type function $G$ on $\Omega$ with poles along $V$}, i.e.\ $G \in PSH(\Omega)$, $G <0$ on $\Omega$, and for some continuous functions $A$ and $B$ on $\Omega$, 
	$$ \log d^2(z, V) + A(z) \geq  G(z) \geq \log d^2(z, V) - B(z). $$
	Then, for every holomorphic funcion $f$ on $V$ with $\int_V |f|^2 e^{-\phi + kB} < +\infty$, there exists a holomorphic function $F$ on $\Omega$ such that $F|_V = f$ and
	$$\int_\Omega |F|^2e^{-\phi} \leq \sigma_k \int_V |f|^2 e^{-\phi + kB}.$$
\end{theorem}

We will prove Theorem \ref{thm:main} following the proof in \cite{BL}. The most important result used in the proof is the following:

\begin{theorem}[{\cite{Ber}}]\label{thm:variation}
	Let $\Omega \subset \mathbb{C}^n$ be a bounded pseudoconvex domain and $\Phi \in PSH \cap C^{\infty}(\overline{\Omega \times \Delta })$ where $\Delta = \{|t| < 1 \} \subset \mathbb{C}$ is the unit disc. For each $t \in \Delta$, we let $\phi_t(z) := \Phi(z,t)$, which is a plurisubharmonic function on $\Omega$.
	Let $\xi \in A^2(\Omega)$ be a bounded linear functional on the Hilbert space of $L^2$ holomorphic functions on $\Omega$ (note that $A^2(\Omega, \phi_t) = A^2(\Omega)$ as vector spaces since by assumption $\Phi$ is bounded). Then the function
	$$t \mapsto \log \|\xi\|_{(A^2_{\Omega, \phi_t})^*} $$
	is subharmonic.
\end{theorem}

\begin{remark}
	It is interesting that, conversely, Theorem \ref{thm:variation} can be proved by the optimal $L^2$-extension theorem \cite{GZ}.
	Furthermore, in \cite{DWZZ}, Theorem \ref{thm:variation} is obtained by the (non-optimal) $L^2$-extension theorem via an $L^2$-theoretic characterization of plurisubharmonic functions.
	According to these results, Theorem \ref{thm:variation} also holds for singular weights.
\end{remark}

\begin{proof}[Proof of Theorem \ref{thm:main}]
We will denote by $A^2(\Omega, \phi)$ the Hilbert space of holomorphic functions $F$ on $\Omega$ with $\int_\Omega |F|^2e^{-\phi}<+\infty$. We can assume that $\phi$ is continuous.
	
First take an arbitrary $L^2$ extension $F^\circ \in A^2(\Omega, \phi)$ of $f$ to $\Omega$ by shrinking domain in advance and using the standard theory of Stein manifolds. (Thus we should consider a sequence of domains $\Omega_j$ approximating $\Omega$ from inside. We will prove the existence of an extension $F_0^{(j)} \in \mathcal{O}(\Omega_j)$ of $f|_{\Omega_j \cap V}$ with uniformly bounded $L^2$-norms. Thus we can extract a convergent subsequence. In the argument below, just for simplicity, we omit the subscript $j$. We will need this approximation sequence of domains again in the last of this proof.)

Consider a short exact sequence of Hilbert spaces:
$$0 \longrightarrow A^2(\Omega, \phi) \cap \mathcal{I}_V \longrightarrow A^2(\Omega, \phi) \longrightarrow  \frac{A^2(\Omega, \phi)}{A^2(\Omega, \phi) \cap \mathcal{I}_V} \longrightarrow 0, $$
where $A^2(\Omega, \phi) \cap \mathcal{I}_V$ denotes the space $A^2(\Omega, \phi) \cap H^0(\Omega, \mathcal{I}_V)$.
Let $F_0$ be the $L^2$-minimum extensions of $f$. The $L^2$ norm $\|F_0\|_{A^2(\Omega, \phi)}$ of $F_0$ is equal to the quotient norm of $F^\circ$ in the space ${A^2(\Omega, \phi)}/{A^2(\Omega, \phi) \cap \mathcal{I}_V}$. Considering dual spaces, we can write the norm as
$$\|F^\circ\|_{A^2/A^2\cap \mathcal{I}_V} = \sup\left\{\frac{| \langle\xi, F^\circ \rangle |}{\|\xi\|_{A^2(\Omega, \phi)^*}}: \xi \in A^2(\Omega, \phi), \langle \xi, h\rangle= 0 \text{ for every }h \in A^2(\Omega, \phi) \cap \mathcal{I}_V \right\}. $$
Let us consider a good class of $\xi$. To do that, fix a smooth function $g$ on $V$ with compact support. Define a linear functional $\xi_g$ on $A^2(\Omega, \phi)$ by
$$\langle \xi_g, h \rangle := \sigma_k\int_V h \overline{g} e^{-\phi+kB}.$$
It can be shown that the set of such functionals $\xi_g$ is a dense subspace of $({A^2(\Omega, \phi)}/{A^2(\Omega, \phi) \cap \mathcal{I}_V})^*$. Thus the supremum above can be written also as
$$\sup_g\frac{| \langle\xi_g, F^\circ \rangle |}{\|\xi_g\|_{A^2(\Omega, \phi)^*}} . $$

For $p \geq 0$ and $t \in \mathbb{C}$ with ${\rm Re}\, t \leq 0$, we let $\phi_{t,p}(z) := \phi(z) + p \max(G(z)-{\rm Re}\, t,0)$. 
Define $A^2_{t,p}:= A^2(\Omega, \phi_{t,p})$. 
For fixed $p \geq 0$, the function $(t, z) \mapsto \phi_{t ,p}(z)$ is
plurisubharmonic. Applying Theorem \ref{thm:variation} to the family $\{A^2(\Omega, \phi_{t,p})\}_t$, we obtain that the function
$$t \mapsto \|\xi_g\|_{(A^2(\Omega, \phi_{t,p}))^*}$$
is subharmonic.
Since $\phi_{t,p}$ depends only on the real part of $t$, it is convex as a function of ${\rm Re}\, t$. From here we assume that $t \in \mathbb{R}_{\leq 0}$. The following lemma describes the limiting behavior of $\|\xi_g\|_{(A^2(\Omega, \phi_{t,p}))^*}$. We will write this norm by $\|\xi_g\|_{t,p}$ for short.

\begin{lemma}[{\cite[Lemma 3.2]{BL}}]
	For fixed $p>0$, it holds that
	$$\|\xi_g \|^2_{t,p} e^{kt} = O(1) $$
	when $t \to -\infty$. In particular, $\log \|\xi_g\|^2_{t,p} + kt$ is convex in $t$ and bounded from above when $t \to -\infty$, thus increasing in $t$.
\end{lemma}

The proof is the same as one in \cite{BL} up to this point. From here we use the idea in \cite{MV}.
We will denote by $F_{t,p}$ the $L^2$-minimal extension of $f$ in $A^2_{t,p} := A^2(\Omega, \phi_{t,p})$. By the same reason as before, we have
$$
\|F_{t,p}\| = \sup_g\frac{| \langle\xi_g, F^\circ \rangle |}{\|\xi_g\|_{t,p}}.
$$
By the lemma before, we have that $e^{-kt} \|F_{t,p}\|^2_{A^2_{t,p}}$ is decreasing in $t$, and thus
$$\|F_0\|^2_{A^2(\Omega, \phi)} \leq e^{-kt} \| F_{t,p} \|^2_{A^2_{t,p}}.$$
Since $F_{t,p}$ is $L^2$-minimal, we have
$$e^{-kt} \| F_{t,p} \|^2_{A^2_{t,p}} \leq e^{-kt} \| F^\circ \|^2_{A^2_{t,p}}.$$
Fix $t<0$ and let $p \to +\infty$. Since $e^{-\phi_{t,p} } \downarrow 1_{\Omega_t} \cdot e^{-\phi}$, Lebesgue's dominated convergence theorem shows that
$$ e^{-kt} \| F^\circ \|^2_{A^2_{t,p}} \to e^{-kt} \|F^\circ\|^2_{A^2(\Omega_t, \phi|_{\Omega_t})}.$$
Next let $t \to -\infty$. Here we need an approximation sequence of domains $\Omega_j \Subset \Omega$ again. We have proved that the $L^2$-minimum extension $F_0^{(j)} \in A^2(\Omega_j, \phi)$ of $f$ satisfies 
$$\|F_0^{(j)}\|^2_{A^2(\Omega_j, \phi)} \leq e^{-kt} \|F^\circ\|^2_{A^2(\Omega_j \cap \{G <t \}, \phi)}.$$
By \cite[Lemma 3.3]{BL}, the limsup of the right-hand side when $t \to -\infty$ is bounded by 
$$\sigma_k \int_{V} |f|^2_{e^{-\phi+ kB}}.$$
Therefore the $L^2$-norms $\|F_0^{(j)}\|_{A^2(\Omega_j, \phi)}$ are uniformly bounded. After taking a subsequence we get a function $F_0^{(\infty)} \in A^2(\Omega, \phi)$ satisfying
$$\|F_0^{(\infty)} \|^2_{A^2(\Omega, \phi)} \leq \sigma_k \int_{V} |f|^2_{e^{-\phi+ kB}}. $$
\end{proof}

\section{General case}
In the same manner, we can prove an optimal $L^2$-estimate for extensions from non-reduced subvarieties in a formulation similar to \cite{Dem}.

The setting is as follows. 
Let $\Omega \Subset \mathbb{C}^n$ be a bounded pseudoconvex domain and $\phi \in PSH(\Omega) \cap C(\Omega)$ be a continuous plurisubharmonic function.
Let $\psi$ be a negative plurisubharmonic function on $\Omega$ with {\it neat analytic singularities}, i.e.\ for each point $x \in \Omega$, there exist a neighborhood $U$ of $x$, a positive number $c>0$, a finite number of holomorphic functions $g_1, \ldots, g_N \in \mathcal{O}(U)$, and a smooth function $u \in C^\infty(U)$ such that
$$\psi(z) = c \log(|g_1(z)|^2 + \cdots + |g_N(z)|^2) + u(z)$$
for $z \in U$. We will regard $\psi$ as a generalization of Green-type functions $G$ appeared in Section 2.

Let $0=m_0<m_1<\cdots < m_p <\cdots$ be the {\it jumping numbers} of $\psi$. By definition, it holds that
$$\mathcal{I}(m_{p-1} \psi) = \mathcal{I}(m\psi) \supsetneq \mathcal{I}(m_p \psi)$$
for every $m \in [m_{p-1}, m_p)$.
In this setting, we have the following theorem:
\begin{theorem}\label{thm:non-reduced}
	Let $f \in H^0(\Omega, \mathcal{I}(m_{p-1}\psi) /\mathcal{I}(m_p\psi))$. Assume that there exists a holomorphic function $F^\circ \in A^2(\Omega, \phi + m_{p-1} \psi)$ such that $F^\circ|_{\mathcal{O}/\mathcal{I}(m_p \psi)} = f$. Then the $L^2$-minimal extension $F_0 \in A^2(\Omega, \phi + m_{p-1} \psi)$ of $f$ satisfies
	$$\int_{\Omega} |F_0|^2e^{-\phi - m_{p-1} \psi} \leq \limsup_{t \to -\infty} e^{-(m_p - m_{p-1})t} \int_{\Omega_t}|F^\circ|^2e^{- \phi - m_{p-1} \psi}, $$
	where $\Omega_t := \{\psi < t \}$.
\end{theorem}

\begin{remark}
	The assumption that $\psi$ has neat analytic singularities may be weakened. For example, let $G$ be the pluricomplex Green function of the polydisc $\Delta^n \subset \mathbb{C}^n$ with a pole at the origin. Then we have $G = \max (\log |z_1|^2, \ldots, \log |z_n|^2)$ and it does not seem to have neat analytic singularities, but we can apply Theorem \ref{thm:main} also to this case. Furthermore, by using the Azukawa indicatrix, we can obtain a sharper estimate than Theorem \ref{thm:main} (see \cite{Hos-Azukawa}). 
\end{remark}

\subsection{Principalization of analytic singularities}
In this subsection, we will explain a technique to analyze multiplier ideals of plurisubharmonic functions with analytic singularities. We will follow the exposition in \cite{Dem}.

Let $\psi$ be a plurisubharmonic function on $\Omega$ with neat analytic singularities. As before, we assume that locally
$$\psi(z) = c \log(|g_1(z)|^2 + \cdots + |g_N(z)|^2) + u(z)$$ 
for some $c>0$, $g_j \in \mathcal{O}(U)$, and $u \in C^\infty(U)$.
By Hironaka's theorem, there exists a modification $\mu: X \to\Omega $ such that the ideal on $X$ generated by functions $g_1 \circ \mu, \ldots, g_N \circ \mu$ is equal to $\mathcal{O}_X(-\Delta)$, where $\Delta$ is a simple normal crossing divisor on $X$. Moreover, we can assume that the zero-divisor of the Jacobian of $\mu$ is contained in $\Delta$.
Taking a suitable coordinate, we can write $(g_j \circ \mu) = (w^a)$ locally for some multi-index $a$.
Then $f \in \mathcal{I}(m \psi)$ if and only if
$$\int_{\mu^{-1}(V) }  \frac{|f \circ \mu(w)|^2|{\rm Jac} (\mu)|^2}{|g \circ \mu(w)|^{2mc}}d\lambda(w) < +\infty$$
in every local coordinate $w$. If we write ${\rm Jac}(\mu) = \beta(w) \cdot w^b$ for non-vanishing holomorphic function $\beta$, this condition can be rewritten as
$$\int_{\mu^{-1}(V) }  \frac{|f \circ \mu(w)|^2|w^b|^2}{|w^a|^{2mc}}d\lambda(w) < +\infty.$$
It is equivalent to the condition
$$f \circ \mu \text{ can be divided by } w^s, \text{ where }s_k = \lfloor m c a_k - b_k\rfloor_+.  $$
Thus, $m>0$ is a jumping number of $\psi$ if and only if
$$m = \frac{b_k+M}{ca_k}$$
for some $k$ and $M \in \mathbb{N}$. For the $p$-th jumping number $m_p$, we will write $s_k^{(p)} := \lfloor m_p c a_k - b_k\rfloor_+$. This means that, for $f \in \mathcal{I}(m_p\psi)$, $\mu^*f $ must have zeros of order $s_k^{(p)}$ along $w_k = 0$.

\subsection{Functionals on the space of $L^2$-holomorphic functions}
In this subsection, we will explain the definition of a functional $\xi_g$ on $A^2(\Omega, \phi + m_{p-1}\psi)$ which will be used in the proof of Theorem \ref{thm:non-reduced}.

We consider the modification $\mu: X \to \Omega$ described in the previous subsection. Fix a local coordinate $w$ on $X$ such that $\Delta = \{w^a=0\}$ for a multi-index $a$.
Let $w_k$ be one of the coordinate functions. We assume that, at $m=m_p$, the jump of ideals $\mathcal{I}(m\psi)$ occurs along $w_k=0$, i.e.\ $m_p = \frac{b_k+M}{ca_k}$. We will write the coordinate as $w = (w', w_k)$.

Fix a smooth function $g \in C^\infty_c((\Delta_{reg})_k)$, where $(\Delta_{reg})_k$ means the intersection of the set $\{w_k = 0 \}$ and the regular locus of $\Delta$. We assume that the support of $g$ is contained in an open set where the coordinate $w$ is defined.

For each $h \in A^2(\Omega, \phi + m_{p-1}\psi)$, we want to define $\xi_g$ by
\begin{equation}
\xi_g (h) := \lim_{t \to -\infty} \int_{t <\psi <t+1}\mu^*h(w) \overline{\widetilde{g}(w)} e^{-\phi-m_p \psi} \mu^* d\lambda_\Omega,\label{eqn:def-xi_g}
\end{equation}
where $\widetilde{g}(w) = g(w')w_k^{s_k^{(p)}-1}$. Note that $\mu^* h(w)$ has zeros of order $s_k^{(p-1)}$ along $w_k=0$. By definition of $s_k^{(p-1)}$, we have that
$$s_k^{(p-1)} \leq  m_{p-1}c a_k-b_k < s_k^{(p-1)} +1.$$
Since we assumed that jump occurs at $m_p$ along $w_k=0$, we also have
$$s_k^{(p)} = m_pca_k-b_k =  s_k^{(p-1)} + 1. $$

Properties of $\xi_g$ are summarized in the following proposition:
\begin{proposition}\label{prop:properties_of_xi_g}
\begin{itemize}
	\item[(1)] The limit (\ref{eqn:def-xi_g}) exists and $\xi_g$ is a bounded linear functional on $A^2(\Omega, \phi + m_{p-1}\psi)$.
	\item[(2)] The set of finite sums of functionals $\xi_g$ is a dense subspace of the dual space of $$\frac{A^2(\Omega, \phi+m_{p-1}\psi)}{A^2(\Omega, \phi + m_{p-1} \psi) \cap \mathcal{I}(m_p \psi)}.$$
	\item[(3)] Let $\phi_{s,q}(z) := \phi(z) + q \cdot\max(\psi(z) - s, 0)$ for $s \leq 0$ and $q>0$. Then $e^{(m_p-m_{p-1})s}\|\xi_g\|^2_{A^2(\Omega, \phi_{s,q}+m_{p-1}\psi)^*} = O(1)$ as $s \to -\infty$.
\end{itemize}
\end{proposition}

\begin{proof}
(1) We will prove the existence of the limit by computing the limit in more explicit terms. What we have to compute is
$$ \lim_{t \to -\infty} \int_{t <\psi <t+1}\mu^*h(w) \overline{\left(g(w')w_k^{s_k^{(p)}-1}\right)} e^{-\phi-m_p \psi} \mu^* d\lambda_\Omega.$$
We can write as $\mu^* h(w) = k(w) \cdot w_k^{s_k^{(p)} - 1}$, where $k(w)$ is a holomorphic function. We can also write as
$$\sum_{j=1}^N|g_j \circ \mu(w)|^2 = \gamma(w)\cdot |w^a|^2,$$
where $\gamma(w)$ is a positive continuous function (this is a square sum of the form $\sum_j \left| g_j \circ \mu(w) / w^a\right|^2$). 
Then,
$$\mu^* \psi = \mu^*(c \log|g|^2 + u) = c \log(\gamma \cdot |w^a|^2) + \mu^* u.$$
In addition, we can write as $\mu^* d\lambda_\Omega = |{\rm Jac}\,\mu|^2 d\lambda(w)= |\beta(w)|^2 |w^b|^2 d\lambda(w)$, where $\beta(w)$ is a non-vanishing holomorphic function. Using these, the integral can be written as
\begin{align*}
&\int_{\substack{w' \in {\rm supp}(g) \\ t < \psi(w',w_k) < t+1}}  k(w) \overline{g(w')} |w_k|^{2(s_k^{(p)}-1)} e^{-\mu^*\phi} \gamma(w)^{-m_pc} |w^a|^{-2m_pc} e^{-m_p \mu^* u(w)} |\beta(w)|^2 |w^b|^2 d\lambda(w)\\
&=\int_{\substack{w' \in {\rm supp}(g) \\ t < \psi(w',w_k) < t+1}}  k(w) \overline{g(w')}  |w_k|^{2(s_k^{(p)}-1)}  |w^{-(m_pca-b)}|^2  e^{-\mu^*\phi} \gamma(w)^{-m_pc}e^{-m_p \mu^* u(w)} |\beta(w)|^2  d\lambda(w).
\end{align*}
The degree of $w_k$ is
$$(s_k^{(p)} -1) - (m_p ca_k - b_k) = -1, $$
thus the integral is like
$$=\int_{\substack{w' \in {\rm supp}(g) \\ t < \psi(w',w_k) < t+1}}  k(w) \overline{g(w')}  |w_k|^{-2} L(w)d\lambda(w),$$
where $L(w)$ is a continuous function.
Consider integration in $w_k$, using Fubini's theorem. Then the domain of integration is 
\begin{equation}
t < c \log\left(\gamma(w) |w^a|^2 \right) + \mu^* u(w) < t+1.\label{eqn:interval}
\end{equation}
Writing as $w^a = (w')^{a'} w_k^{a_k}$, we have that
$$t - c\log (\gamma(w)|(w')^{a'}|^2) - \mu^*u(w) < ca_k \log|w_k|^2 < t - c\log (\gamma(w)|(w')^{a'}|^2) - \mu^*u(w) + 1.$$
Thus we can write (\ref{eqn:interval}) as
$$t - \diamondsuit (w')-\epsilon(w) < ca_k \log |w_k|^2 < t - \diamondsuit(w') + 1 + \epsilon'(w),$$
where $\diamondsuit(w')$ is a continuous function and $\epsilon, \epsilon'$ is a small error term converging to $0$ when $t \to -\infty.$
Then, by a straightforward computation, one can obtain that the limit exists and is
$$\frac{\pi}{ca_k} \int_{w' \in {\rm supp}(g)} k(w',0) \overline{g(w')} L(w',0) d\lambda(w').$$
This shows the linearity of $\xi_g$. Boundedness of $\xi_g$ can be proved in the same way as in the proof of (3), so this part is postponed.

(2) What we need to prove is that, for each $h \in A^2(\Omega, \phi+m_{p-1}\psi)$, if $\xi_g(h) = 0$ for every $g$ then $h \in H^0(\Omega, \mathcal{I}(m_p)\psi)$. By the calculation in (1), $\xi_g(h) = 0$ for every $g$ implies that $h$ has zeros of order $s_k^{(p)}$ along $\{w_k=0\}$ for every $w_k$. This implies that $h \in \mathcal{I}(m_p \psi)$.

(3) The norm of $\xi$ is
\begin{align*}
\|\xi\|^2_{A^2(\Omega, \phi_{s,q} + m_{p-1}\psi)^*} &= \sup_h \frac{|\langle \xi_g, h \rangle |^2}{\|h\|^2_{A^2(\Omega, \phi_{s,q} + m_{p-1}\psi)}}\\
& = \sup_h \frac{|\langle \xi_g, h \rangle |^2}{\int_\Omega |h|^2 e^{-\phi_{s,q} - m_{p-1} \psi}}.
\end{align*}
By Cauchy-Schwarz inequality, the numerator can be estimated as
\begin{align*}
&\lim_{t \to -\infty} \left(\left|\int_{\substack{w' \in {\rm supp}(g) \\ t < \psi(w',w_k) < t+1}} \mu^* h \cdot \overline{\widetilde{g}} \cdot e^{-\phi - m_p \psi} \mu^* d\lambda(w)\right|\right)^2\\
&\leq \lim_{t \to -\infty} \left(\int_{\substack{w' \in {\rm supp}(g) \\ t < \psi(w',w_k) < t+1}}  |\mu^* h|^2 e^{-\phi - m_p \psi} \mu^* d\lambda \right) \cdot \left(\int_{\substack{w' \in {\rm supp}(g) \\ t < \psi(w',w_k) < t+1}}  |\widetilde{g}|^2 e^{-\phi - m_p \psi} \mu^* d\lambda \right).
\end{align*}
The second integral is independent of $h$ and can be bounded by a constant, say $C_g$.
We want to prove the following estimate:
\begin{equation}
\lim_{t \to -\infty}\int_{\substack{w' \in {\rm supp}(g) \\ t < \psi(w',w_k) < t+1}}  |\mu^* h|^2 e^{-\phi - m_p \psi} \mu^* d\lambda \leq Ce^{-(m_p-m_{p-1})s}\int_{\substack{w' \in U \\ \psi(w',w_k) < s}} |\mu^* h|^2 e^{-\phi - m_{p-1}\psi}  \mu^* d\lambda,\label{eqn:limit_estimate}
\end{equation}
where $U \subset \Delta_k$ is a neighborhood of ${\rm supp}(g)$. Once it is proved, the right-hand side can be bounded as
\begin{align*}
(RHS) &\leq C e^{-(m_p-m_{p-1})s} \int_{\psi< s} |h|^2 e^{-\phi - m_{p-1}\psi}d\lambda(z)\\
&\leq C e^{-(m_p-m_{p-1})s} \int_\Omega |h|^2 e^{-\phi_{s,q} - m_{p-1}\psi}.
\end{align*}
This gives the desired conclusion. 

Let us prove (\ref{eqn:limit_estimate}). 
Using $\mu^* h = w_k^{s_k^{(p)}} k(w)$, $\mu^*\psi = c(\log (\gamma(w) \cdot |w^\alpha|^2)) + \mu^*u(w)$, and $\mu^* \lambda = \beta(w) w^b d \lambda(w)$, each integration can be rewritten as:
$$\int_{t<\psi < t+1} |w_k|^{2(s_k^{(p)}-1)}  |k|^2 e^{-2\phi} \gamma(w)^{-m_pc} |w^a|^{-2m_pc} e^{-m_p \mu^* u} |\beta|^2 |w^b|^2 d\lambda(w),\text{ and}$$
$$\int_{\psi < s} |w_k|^{2(s_k^{(p)}-1)}  |k|^2 e^{-2\phi} \gamma(w)^{-m_{p-1}c} |w^a|^{-2m_{p-1}c} e^{-m_{p-1} \mu^*a u} |\beta|^2 |w^b|^2 d\lambda(w).$$
Fix $w'$ and consider integrations in $w_k$. Then essentially it is enough to prove that
\begin{align*}
&\lim_{t \to -\infty}\int_{\frac{t-\diamondsuit}{ca_k} < \log |w_k|^2 < \frac{t-\diamondsuit+1}{ca_k}} |w_k|^{2(s_k^{(p)}-1-m_pca_k+b_k)} d\lambda(w_k) \\
&\hspace{2cm}\leq C(e^{-(m_p-m_{p-1}s')})
\int_{\log |w_k|^2 < \frac{s-\diamondsuit}{ca_k}} |w_k|^{2(s_k^{(p)}-1-m_{p-1}ca_k+b_k)} d\lambda(w_k).
\end{align*}
Since $s_k^{(p)}-m_pca_k+b_k = 0$, it is equivalent to
\begin{align*}
&\lim_{t \to -\infty}\int_{\frac{t-\diamondsuit}{ca_k} < \log |w_k|^2 < \frac{t-\diamondsuit+1}{ca_k}} |w_k|^{-2} d\lambda(w_k) \\
& \hspace{2cm}\leq C(e^{-(m_p-m_{p-1}s')})
\int_{\log |w_k|^2 < \frac{s-\diamondsuit}{ca_k}} |w_k|^{2((m_p-m_{p-1})ca_k-1 )} d\lambda(w_k).
\end{align*}
This can be proved by a simple computation.
\end{proof}

\subsection{Proof of Theorem \ref{thm:non-reduced}}
As in the proof of Theorem \ref{thm:main}, the minimum $L^2$-norm of $L^2$-extension is equal to the quotient norm of $F^\circ$ in the space $\displaystyle\frac{A^2(\Omega, \phi + m_{p-1}\psi)}{A^2(\Omega, \phi + m_{p-1}\psi) \cap \mathcal{I}(m_{p}\psi)}$. By considering the dual, this equals to
$$\sup_g \frac{|\langle \xi_g, F^\circ\rangle|}{\|\xi_g\|_{A^2(\Omega, \phi+m_{p-1}\psi)^*}}, $$
where $\xi_g$ is the functional defined in the previous subsection.

For $s<0$ and $q > 0$, let $\phi_{s,q}(z) := \phi(z) + q \max(G(z)-s, 0)$.
Denote by $F_{s,q}$ the $L^2$-minimal extension of $f$ in the space $A^2(\Omega, \phi_{s,q} + m_{p-1}\psi)$.
Then, by the same reason, 
$$\|F_{s,q}\|_{A^2(\Omega, \phi_{s,q} + m_{p-1}\psi)} = \sup_g \frac{|\langle \xi_g, F^\circ\rangle|}{\|\xi_g\|_{A^2(\Omega, \phi_{s,q}+m_{p-1}\psi)^*}}.$$
By a singular version of Theorem \ref{thm:variation}, we have that $\log \|\xi_g\|_{A^2(\Omega, \phi_{s,q}+m_{p-1}\psi)^*}$ is convex in $s$. (Here we have to consider singular weights in Theorem \ref{thm:variation}. In this situation, we should prove Theorem \ref{thm:variation} in advance by the optimal $L^2$-extension theorem (Theorem \ref{thm:main}) as in \cite{GZ}. Also see \cite{Hos-jet}.)
By Proposition \ref{prop:properties_of_xi_g} (3), we also have that ${(m_p-m_{p-1})s} + \log \|\xi_g\|^2_{A^2(\Omega, \phi_{s,q}+m_{p-1}\psi)^*}$ is bounded from above when $s \to -\infty$, and thus increasing in $s$.
Therefore, $e^{-(m_p-m_{p-1})s} \|F_{s,q}\|^2_{A^2(\Omega, \phi_{s,q} + m_{p-1}\psi)}$ is decreasing in $s$. Then we have the following estimates:
\begin{align*}
\|F_0\|^2_{A^2(\Omega, \phi + m_{p-1}\psi)}  \leq&\, e^{-(m_p-m_{p-1})s} \|F_{s,q}\|^2_{A^2(\Omega, \phi_{s,q} + m_{p-1}\psi)} \\
 \leq&\, e^{-(m_p-m_{p-1})s} \|F^\circ\|^2_{A^2(\Omega, \phi_{s,q} + m_{p-1}\psi)} \\
 \overset{q \to +\infty}{\longrightarrow} &\, e^{-(m_p-m_{p-1})s} \|F^\circ\|^2_{A^2(\Omega_s, \phi + m_{p-1}\psi)}
\end{align*}
This completes the proof.

{\bf Acknowledgment. } This work is supported by the Grant-in-Aid for JSPS Fellows $\sharp$19J00473.

	
	\bibliographystyle{plain}

\begin{thebibliography}{10}
		
		\bibitem[Ber]{Ber}
		B.~Berndtsson,
		\newblock {\em Curvature of vector bundles associated to holomorphic fibrations,}
		\newblock {Ann. of Math. (2)}, 169(2):531--560, 2009.
		
		\bibitem[BL]{BL}
		B.~Berndtsson and L.~Lempert,
		\newblock {\em A proof of the {O}hsawa-{T}akegoshi theorem with sharp estimates,}
		\newblock {J. Math. Soc. Japan}, 68(4):1461--1472, 2016.

		\bibitem[Blo]{Blo}
        Z.~B{\l}ocki,
        \newblock {\em Suita conjecture and the Ohsawa-Takegoshi extension theorem,}
        \newblock {Invent. Math.}, 193(1):149--158, 2013.

		\bibitem[Dem]{Dem}
		J.-P.~Demailly,
		\newblock {\em Extension of holomorphic functions and cohomology classes from non reduced analytic subvarieties,}
		\newblock {Geometric Complex Analysis. Springer Proceedings in Mathematics \& Statistics, vol 246. Springer, Singapore.}
		
		\bibitem[DWZZ]{DWZZ}
		F.~Deng, Z.~Wang, L.~Zhang, X.~Zhou,
		\newblock {\em New characterizations of plurisubharmonic functions and positivity of direct image sheaves,}
		arXiv:1809.10371.
		
		\bibitem[GZ]{GZ}
		Q.~Guan and X.~Zhou,
		\newblock {\em A solution of an {$L^2$} extension problem with an optimal estimate
			and applications,}
		\newblock {Ann. of Math. (2)}, 181(3):1139--1208, 2015.
		
		\bibitem[Hos1]{Hos-jet}
		G.~Hosono,
		\newblock {\em The Optimal jet extension of Ohsawa-Takegoshi type,}
		to appear in Nagoya Math. J.
				
		\bibitem[Hos2]{Hos-Azukawa}
		G.~Hosono,
		\newblock {\em Subharmonic variation of Azukawa pseudometrics for balanced domains,}
		arXiv:1811.07154.
		
		\bibitem[MV]{MV}
		J.~McNeal and D.~Varolin,
		\newblock {\em Extension of Jets With $ L^ 2$ Estimates, and an Application,}
		arXiv:1707.04483.
			
		\bibitem[OT]{OT}
		T.~Ohsawa and K.~Takegoshi,
		\newblock {\em On the extension of {$L^2$} holomorphic functions,}
		\newblock {Math. Z.}, 195(2):197--204, 1987.	
		
		\bibitem[ZZ]{ZZ}
		X.~Zhou and L.~Zhu,
		\newblock {\em Extension of cohomology classes and holomorphic sections defined on subvarieties,}
		arXiv:1909.08822.	
	\end{thebibliography}

\end{document}